\documentclass[12pt,reqno]{amsart}
\usepackage{amsfonts}
\usepackage{amssymb,color}
\usepackage{amssymb}

\def\normo#1{\left\|#1\right\|}

\def\aabs#1{\left|#1\right|}
\def\brk#1{\left(#1\right)}

\def\norm#1{\|#1\|}
\def\jb#1{\langle#1\rangle}
\def\wt#1{\widetilde{#1}}

\newcommand{\N}{{\mathbb N}}

\newcommand{\R}{{\mathbb R}}
\newcommand{\C}{{\mathbb C}}
\newcommand{\Z}{{\mathbb Z}}
\newcommand{\ft}{{\mathcal{F}}}
\newcommand{\Hl}{{\mathcal{H}}}

\newcommand{\les}{{\lesssim}}
\newcommand{\ges}{{\gtrsim}}

\newcommand{\supp}{{\mbox{supp}}}
\newcommand{\sgn}{{\mbox{sgn}}}

\newcommand{\Lr}{{\mathcal{L}}}

\def\jb#1{\langle#1\rangle}
\def\norm#1{\|#1\|}
\def\normo#1{\left\|#1\right\|}

\def\wt#1{\widetilde{#1}}

\def\aabs#1{\left|#1\right|}

\newcommand{\cir}{\mathbb{S}}

\newcommand{\al}{\alpha}

\newcommand{\e}{\varepsilon}

\newcommand{\x}{\xi}
\newcommand{\y}{\eta}

\newcommand{\De}{\Delta}
\newcommand{\Des}{\Delta_\omega}
\newcommand{\Om}{\Omega}

\newcommand{\p}{\partial}
\newcommand{\na}{\nabla}

\newcommand{\lec}{\lesssim}

\newcommand{\I}{\infty}

\newcommand{\LR}[1]{{\langle #1 \rangle}}

\newcommand{\EQ}[1]{\begin{equation}\begin{split} #1 \end{split}\end{equation}}
\newcommand{\Del}[1]{}
\newcommand{\CAS}[1]{\begin{cases} #1 \end{cases}}

\newcommand{\pt}{&}
\newcommand{\pr}{\\ &}

\numberwithin{equation}{section}

\newtheorem{thm}{Theorem}[section]
\newtheorem{cor}[thm]{Corollary}
\newtheorem{lem}[thm]{Lemma}
\newtheorem{prop}[thm]{Proposition}

\theoremstyle{remark}
\newtheorem{rem}{Remark}

\theoremstyle{remark}

\theoremstyle{definition}

\begin{document}

\title[Strichartz estimates for Schr\"odinger equation]{Sharp spherically averaged Strichartz estimates for the Schr\"odinger
equation}
\author[Z. Guo]{Zihua Guo}

\address{LMAM, School of Mathematical Sciences, Peking
University, Beijing 100871, China \& Department of Mathematics, The
Australian National University, Canberra, ACT 0200, Australia}

\email{zihuaguo@math.pku.edu.cn}

\begin{abstract}
We prove generalized Strichartz estimates with weaker angular
integrability for the Schr\"odinger equation. Our estimates are
sharp except some endpoints. Then we apply these new estimates to
prove the scattering for the 3D Zakharov system with small data in
the energy space with low angular regularity. Our results improve
the results obtained recently in \cite{GLNW}.
\end{abstract}

\maketitle

\tableofcontents

\section{Introduction}

In this paper, we continue the previous work \cite{GLNW} to study
the generalized Strichartz estimates for the following
Schr\"odinger-type dispersive equations
\begin{align}\label{eq:schr}
i\partial_{t}u+D^a u=0,\ u(0,x)=f(x)
\end{align}
where $u(t,x):\R\times \R^{d}\to \C$, $D=\sqrt{-\Delta}$, $a>0$. We
are mainly concerned with following estimates
\begin{align}\label{eq:striestwea}
\norm{e^{itD^a}P_0f}_{L_t^q\Lr_{\rho}^pL_\omega^2}\les
\norm{f}_{L_x^2}
\end{align}
where $\widehat{P_0 f}\approx 1_{|\xi|\sim 1}\hat{f}$ (See the end
of this section for the precise definition). Here the norm
$L_t^q\Lr_{\rho}^pL_\omega^s$ for function $u(t,x)$ on $\R\times
\R^d (d\geq 2)$ is defined as follows
\[\norm{u}_{L_t^q\Lr_{\rho}^pL_\omega^s}=\brk{\int_\R\bigg[\int_0^\infty\bigg|\int_{\cir^{d-1}}
|u(t,\rho
x')|^sd\omega(x')\bigg|^{\frac{p}{s}}\rho^{d-1}d\rho\bigg]^{\frac{q}{p}}dt}^{1/q}.\]
The purpose of this paper is to study the sharp range for $(q,p)$
such that the estimate \eqref{eq:striestwea} holds. Then we apply
these estimates to the 3D Zakharov system.

Two typical examples are of particular interest, one is the wave
equation ($a=1$), the other is the Schr\"odinger equation ($a=2$).
The space time estimates which are called Strichartz estimates
address the estimates
\begin{align}\label{eq:striest}
\norm{e^{itD^a}P_0f}_{L_t^qL_x^p}\les \norm{f}_{L^2}.
\end{align}
Strichartz \cite{Str} first proved \eqref{eq:striest} for the case
$q=p$  and then the estimates were substantially extended by various
authors, e.g. \cite{GiV95,LS} for $a=1$, and \cite{GV89,Yaj} for
$a=2$. It is now well-known (see \cite{KT}) that the optimal range
for \eqref{eq:striest} is the admissible condition: if $a=1$,
\begin{align}\tag{WA}
2\leq q,p\leq \infty,\,\frac{1}{q}\leq
\frac{d-1}{2}(\frac{1}{2}-\frac{1}{p}),\, (q,p,d)\ne (2,\infty,3);
\end{align}
and if $a\ne 1$,
\begin{align}\tag{SA}
2\leq q,p\leq \infty,\,\frac{1}{q}\leq
\frac{d}{2}(\frac{1}{2}-\frac{1}{p}),\, (q,p,d)\ne (2,\infty,2).
\end{align}
However, if $d\geq 2$ and $f$ is radial, then \eqref{eq:striest}
holds for a wider range of $(q,p)$ (for example, see \cite{KlMa93,
Sogge, Shao}). For the wave equation ($a=1$), the optimal range for
\eqref{eq:striest} under radial symmetry assumption is (see
\cite{KlMa93,Sogge,Ster}, see also \cite{GuoWang})
\begin{align}\tag{RWA}
2\leq q,p\leq \infty,\,\frac{1}{q}<(d-1)(\frac{1}{2}-\frac{1}{p})
\mbox{ or } (q,p)=(\infty,2).
\end{align}
For the Schr\"odinger-type equation ($a\ne 1$), it was known that
\eqref{eq:striest} holds under radial symmetry assumption if the
following condition holds:
\begin{align}\tag{RSA}\label{eq:RSA}
2\leq q,p\leq \infty,\,\frac{1}{q}\leq
(d-\frac{1}{2})(\frac{1}{2}-\frac{1}{p}),\, (q,p)\ne
(2,\frac{4d-2}{2d-3}).
\end{align}
The range (RSA) is optimal except the radial endpoint $(q,p)=
(2,\frac{4d-2}{2d-3})$ which still remains open. (RSA) was first
obtained in \cite{GuoWang} except some endpoints improving the
results in \cite{Shao} and the remaining endpoint estimates were
later obtained in \cite{CL, Ke} independently.

There are two kinds of analogue results in the non-radial case. The
first is to consider the estimate with additional angular regularity
(See the end of this section for the definition of $H_\omega^{0,s}$)
\begin{align}\label{eq:striestang}
\norm{e^{itD^a}P_0f}_{L_t^qL_x^p}\les \norm{f}_{H_\omega^{0,s}}
\end{align}
in which some angular regularity is traded off by the extension of
admissible range, see \cite{Ster,JWY,CL}. The second one is to
consider the estimate \eqref{eq:striestwea} that we study in this
paper. From the viewpoint of application, the estimate
\eqref{eq:striestwea} works better than \eqref{eq:striestang},
because there is no loss of angular regularity. The spherically
averaged Strichartz norm was used in \cite{Tao2} to obtain the
endpoint case of Strichartz estimate for 2D Schr\"odinger (see
\cite{MaNaNaOz05} for 3D wave equation). For the wave equation
($a=1$), it was known that \eqref{eq:striestwea} also holds for
$(q,p)$ satisfying (RWA) (see \cite{SmSoWa12, JWY, CL}). If $a\ne
1$, \cite{GLNW} showed that \eqref{eq:striestwea} holds for $(q,p)$
satisfying (RWA), improving the results in \cite{JWY}. Moreover,
when $a>1$, \cite{GLNW} also showed \eqref{eq:striestwea} holds for
$(q,p)$ belonging to a wider range than (RWA), but the sharp range
is unknown.

The main result of this paper is

\begin{thm}\label{thm1}
Let $a>1$. Then \eqref{eq:striestwea} holds if $(q,p)$ satisfies

(i) $d=2$:
\begin{align}
2\leq q,p\leq \infty,\,\frac{1}{q}<
(d-\frac{1}{2})(\frac{1}{2}-\frac{1}{p})\mbox{ or }
(q,p)=(\infty,2).
\end{align}

(ii) $d\geq 3$: $(q,p)$ satisfies \eqref{eq:RSA}.
\end{thm}

By the theorem above, we see that the optimal range for
\eqref{eq:striestwea} is obtained except the endpoint line
$\frac{1}{q}=(d-\frac{1}{2})(\frac{1}{2}-\frac{1}{p})$ for $d=2$,
and the endpoint $(q,p)=(2,\frac{4d-2}{2d-3})$ for $d\geq 3$. The
basic ideas of proving Theorem \ref{thm1} are the same as in
\cite{GLNW}, namely to do the space dyadically localized estimates
by exploiting the decay and oscillatory effect of a family of Bessel
functions uniformly. The key differences in this paper are: 1) we
prove better uniform estimates for Bessel functions (indeed, we give
a uniform expansion) in the transitive region; 2) we treat the
oscillatory integral operator related to the Bessel function in
finer scale so that we can catch more subtle oscillatory effects; 3)
To get the endpoint for $d\geq 3$, we exploit some ``almost
orthogonality" to overcome some logarithmic summation difficulty.

Besides its own interest, \eqref{eq:striestwea} plays important
roles in the nonlinear problems, e.g. in \cite{GLNW} where the
authors proved scattering for the 3D Zakharov system for small data
in the energy space with one additional angular regularity. In this
paper, we use the new estimates in Theorem \ref{thm1} to improve the
angular regularity. Consider the 3D Zakharov system:
\begin{align}\label{eq:Zak}
 \CAS{ i\dot u - \De u = nu,\\
   \ddot n/\al^2 - \De n = -\De|u|^2,}
\end{align}
with the initial data
\begin{align}
u(0,x)=u_0,\, n(0,x)=n_0,\,\dot n(0,x)=n_1,
\end{align}
where $(u,n)(t,x):\R^{1+3}\to\C\times\R$, and $\al>0$ denotes the
ion sound speed. The system was introduced by Zakharov \cite{Zak} as
a mathematical model for the Langmuir turbulence in unmagnetized
ionized plasma. It preserves $\|u(t)\|_{L^2_x}$ and the energy \EQ{
 E=\int_{\R^3}|\na u|^2+\frac{|D^{-1}\dot n|^2/\al^2+|n|^2}{2}-n|u|^2 dx}
The natural energy space for initial data is \EQ{\label{eq:indata}
 (u_0,n_0,n_1)\in H^1(\R^3)\times L^2(\R^3)\times \dot H^{-1}(\R^3).}
The Zakharov system has been extensively studied, see
\cite{BoCo,GTV,KPV,MN2,GTV,BHHT,GTV,BeHe,BoCo,Takaoka,Kishi,SW,OT,MN,Merle,Shimo,GV,OT2,GN,GNW},
and the introduction of \cite{GLNW}.

Since global well-posedness for \eqref{eq:Zak} with small data in
the energy space was proved by Bourgain-Colliander \cite{BoCo}, the
long time behavior of the solutions has been a very interesting
problem. For this problem, the first result was obtained in
\cite{GN} that scattering holds for small energy data under radial
symmetry assumption. Later, global dynamics below ground state in
the radial case was obtained in \cite{GNW}. In the non-radial case,
for data with sufficient regularity and decay, and with suitable
small norm, scattering was obtained in \cite{HPS}. In \cite{GLNW},
the authors proved scattering for small data in the energy space
with one order angular regularity. In this paper we prove

\begin{thm}\label{thm2}
Let $s>3/4$. Assume $\norm{(u_0,n_0,n_1)}_{H_\omega^{1,s}\times
H_\omega^{0,s}\times \dot H_\omega^{-1,s}}=\e$ for $\e>0$
sufficiently small. Then the global solution $(u,n)$ to
\eqref{eq:Zak} belongs to $C_t^0H_\omega^{1,s}\times
C_t^0H_\omega^{0,s}\cap C_t^1\dot{H}_\omega^{-1,s}$, and scatters in
this space.
\end{thm}

Theorem \ref{thm2} was proved in \cite{GLNW} for $s=1$. We improve
the angular regularity to $s>3/4$ by using the new estimates of
Theorem \ref{thm1}. To deal with the fractional derivative on the
sphere, we transfer it to the fractional derivative on $SO(3)$ (see
the appendix). We remark that $s>3/4$ reaches a limitation of our
method. To remove the angular regularity, new ideas should be
developed.

\subsubsection*{Notations} Finally we close this section by listing the notation.

\noindent $\bullet$ We denote $\N=\Z\cap [0,\infty)$, $\N^*=\Z\cap
(0,\infty)$.

\noindent $\bullet$  $\ft(f)$ and $\widehat{f}$  denote the Fourier
transform of $f$, $\hat f(\xi)=\int_{\R^n}e^{-ix\cdot\xi}f(x)dx$.
For $a\geq 1$, $S_a(t)=e^{itD^a}=\ft^{-1}e^{it|\xi|^a}\ft$.

\noindent $\bullet$ $\eta: \R\to [0, 1]$ is an even, non-negative
smooth function which is supported in $\{\xi:|\xi|\leq 8/5\}$ and
$\eta\equiv 1$ for $|\xi|\leq 5/4$. For $k\in \Z$,
$\chi_k(\xi)=\eta(\xi/2^k)-\eta(\xi/2^{k-1})$ and $\chi_{\leq
k}(\xi)=\eta(\xi/2^k)$.

\noindent $\bullet$ $P_k, P_{\leq k}$ are defined on $L^2(\R^d)$ by
$\widehat{P_ku}(\xi)=\chi_k(|\xi|)\widehat{u}(\xi),\,\widehat{P_{\leq
k}u}(\xi)=\chi_{\leq k}(|\xi|)\widehat{u}(\xi)$.
\smallskip

\noindent $\bullet$   $\Des$ denotes the Laplace-Beltrami operator
on the unite sphere $\cir^{d-1}$ endowed with the standard metric
$g$ measure $d\omega$ and $\Lambda_\omega=\sqrt{1-\Delta_\omega}$.
Denote $L_\omega^p=L_\omega^p(\cir^{d-1})=L^p(\cir^{d-1}:d\omega)$,
$\Hl_p^s=\Hl_p^s(\cir^{d-1})=\Lambda_\omega ^{-s}L_\omega^p$.

\smallskip

\noindent $\bullet$ $L^p(\R^d)$ denotes the usual Lebesgue space,
and $\Lr^p(\R^+)=L^p(\R^+:r^{d-1}dr)$.

\noindent $\bullet$  $\Lr_r^pL_\omega^q $ and $\Lr_r^p\Hl^s_q$ are
Banach spaces defined  by the  following norms
\[\norm{f}_{\Lr_{r}^pL_\omega^q}=\big\|{\norm{f(r\omega)}_{L_\omega^q}}\big\|_{\Lr_{r}^p},\
\norm{f}_{\Lr_{r}^p\Hl^s_q}=\big\|{\norm{f(r\omega)}_{\Hl^s_q}}\big\|_{\Lr_{r}^p}.\]

\smallskip

\noindent $\bullet$  $H^s_p$, $\dot{H}_p^s$ ($B^s_{p,q}$,
$\dot{B}^s_{p,q}$) are the usual Sobolev (Besov) spaces on $\R^d$.

\noindent $\bullet$  $\dot{B}^s_{(p,q),r}$ denotes the Besov-type
space given  by the  norm
\[\norm{f}_{\dot{B}^s_{(p,q),r}}=(\sum_{k\in \Z}2^{ksr}\norm{P_kf}_{\Lr_r^pL_\omega^q}^r)^{1/r}.\]

\noindent $\bullet$  $0\leq \alpha\leq 1$, $H^{s,\alpha}_{p,\omega}$
is the space with the norm
$\norm{f}_{H^{s,\alpha}_{p,\omega}}=\norm{\Lambda_\omega^{\alpha}f}_{H^s_p}$,
and the spaces $\dot{H}^{s,\alpha}_{p,\omega}$,
$B^{s,\alpha}_{p,q,\omega}$, $\dot{B}^{s,\alpha}_{p,q,\omega}$, and
$\dot{B}^{s,\alpha}_{(p,q),r,\omega}$ are defined similarly.

\noindent $\bullet$  For simplicity, we denote
$H^{s,\alpha}_{\omega}=H^{s,\alpha}_{2,\omega}$,
$\dot{H}^{s,\alpha}_{\omega}=\dot{H}^{s,\alpha}_{2,\omega}$,
$B^{s,\alpha}_{p,\omega}=B^{s,\alpha}_{p,2,\omega}$,
$\dot{B}^{s,\alpha}_{p,\omega}=\dot{B}^{s,\alpha}_{p,2,\omega}$,
$\dot{B}^{s,\alpha}_{(p,q),\omega}=\dot{B}^{s,\alpha}_{(p,q),2,\omega}$.

\smallskip

\noindent $\bullet$ Let $X$ be a Banach space on $\R^d$. $L_t^qX$
denotes the space-time space on $\R\times \R^d$ with the norm
$\norm{u}_{L_t^qX}=\big\|\norm{u(t,\cdot)}_X\big\|_{L_t^q}$.

\section{Uniform estimates for Bessel functions}

To prove Theorem \ref{thm1}, we need to deal with a family of Bessel
functions which are defined by
\begin{align*}
J_\nu(r)=\frac{(r/2)^\nu}{\Gamma(\nu+1/2)\pi^{1/2}}
\int_{-1}^1e^{irt}(1-t^2)^{\nu-1/2}dt, \ \ \nu>-1/2.
\end{align*}
In this section, we study the uniform properties for the Bessel
functions $J_\nu(r)$ with respect to the order $\nu$ by some
dedicate stationary phase analysis. Consider the oscillatory
integral
\[I(\lambda)=\int_{-\infty}^\infty e^{i\lambda \phi(x)}a(x)dx, \quad \lambda\in \R\]
where $\phi\in C^\infty(\R)$ and $a\in C_0^\infty(\R)$. We are
interested in the behavior of $I(\lambda)$ as $\lambda\to \infty$.
In order to apply it to the estimate for Bessel function, we need to
track the dependence of $\phi$, $a$, since we allow $\phi$ depends
on $\lambda$. There is a classical useful Van der Corput lemma (see
\cite{Stein2}):

\begin{lem}[Van der Corput]\label{lem:staph}
Suppose $\phi$ is real-valued and smooth in $(a,b)$, and that
$|\phi^{(k)}(x)|\geq 1$ for all $x\in (a,b)$. Then
\[\aabs{\int_a^b e^{i\lambda \phi(x)}\psi(x)dx}\leq c_k \lambda^{-1/k}\bigg[|\psi(b)|+\int_a^b|\psi'(x)|dx\bigg]\]
holds when (i) $k\geq 2$, or (ii) $k=1$ and $\phi'(x)$ is monotonic.
Here $c_k$ is a constant depending only on $k$.
\end{lem}

By the Van der Corput lemma and the Schl\"{a}fli's integral
representation of Bessel function (see p. 176, \cite{Bess}):
\begin{align}\label{eq:Besselint}
J_\nu(r)=&\frac{1}{2\pi}\int_{-\pi}^\pi e^{i(r\sin x-\nu
x)}dx-\frac{\sin(\nu\pi)}{\pi}\int_0^\infty
e^{-\nu\tau-r\sinh \tau}d\tau \nonumber\\
:=&J_\nu^M(r)-J_\nu^E(r),
\end{align}
we can prove

\begin{lem}\label{lem:Bessdecay} Assume $r,\nu>10$. Then we have
\begin{align}
|J_\nu(r)|+|J_\nu'(r)|\leq& Cr^{-1/3}(1+r^{-1/3}|r-\nu|)^{-1/4}.
\end{align}
\end{lem}
\begin{proof}
It's obvious that $|J_\nu^E(r)|\les (r+\nu)^{-1}$, then it remains
to consider $J_\nu^M$. We may assume $\nu\les r$. If
$|r-\nu|>r^{1/3}$, see the proof of Lemma 3.2 in \cite{GLNW}. If
$|r-\nu|\leq r^{1/3}$, it follows from Lemma \ref{lem:staph} with
$k=2,3$.
\end{proof}

\begin{lem}\label{lem:Bessdecay2}
Assume $\nu\in \N$, $\nu>r+\lambda$, and
$\lambda>r^{\frac{1}{3}+\e}$ for some $\e>0$. Then for any $K\in \N$
\begin{align}
|J_\nu(r)|+|J_\nu'(r)|\leq C_{K,\e}r^{-K\e}.
\end{align}
\end{lem}
\begin{proof}
We only need to estimate $J_\nu(r)$, since $2
J_\nu'(r)=J_{\nu-1}(r)-J_{\nu+1}(r)$ (see p. 45 (2) in \cite{Bess}).
If $\nu\geq 2r$, then integrating by part $K$ times we can get
$|J_\nu(r)|\les \nu^{-K}\les r^{-K}$. Now we assume $\nu\leq 2r$.
Let $\phi(x)=r\sin x-\nu x$. Then $\phi'(x)=r\cos x-\nu$. By the
assumption $|\phi'(x)|>r^{\frac{1}{3}+\e}$. Define the operator
$D_\phi$ by
\[D_\phi(f)=-\p_x(\phi'(x)^{-1}f).\]
Then \[J_\nu(r)=\frac{1}{2\pi}\int_{-\pi}^\pi
\phi'(x)^{-1}\p_x[e^{i(r\sin x-\nu
x)}]dx=\frac{1}{2\pi}\int_{-\pi}^\pi e^{i(r\sin x-\nu
x)}D_\phi^K(1)dx.\] It suffices to show
\begin{align}\label{eq:Jbound}
|D_\phi^K(1)|\leq C_{K,\e} r^{-K\e}.
\end{align}
We prove \eqref{eq:Jbound} by induction on $K$. For $K=0$, the bound
is trivial. Now we consider $K=1$. We have
\[\p_x(\phi'(x)^{-1})=\phi'(x)^{-1}\frac{r\sin x}{\phi'(x)}.\]
Let $g(x)=\frac{r\sin x}{\phi'(x)}$. By calculus we see
\begin{align}\label{eq:gbound}
|g(x)|\leq |g(\arccos \frac{r}{\nu})|=\frac{r}{\sqrt{\nu^2-r^2}}\les
r^{\frac{1}{3}}.
\end{align}
Thus by assumption the case $K=1$ is proved. Now we assume
\eqref{eq:Jbound} holds for $K$ by using the bound
\eqref{eq:gbound}. If $g$ is not a factor of $D_\phi^{K}(1)$, then
\[|D_\phi^{K+1}(1)|\sim |D_\phi^{K}(1)\phi'(x)^{-1}g(x)|\les r^{-(K+1)\e}.\]
If $g$ is a factor of $D_\phi^{K}(1)$, then
\[D_\phi^{K+1}(1)=CD_\phi^{K}(1)\phi'(x)^{-1}g(x)+\tilde C G\]
where $G$ is given by $D_\phi^K(1)$ but with one factor $g$ replaced
by $\phi'(x)^{-1}\frac{r\cos x}{\phi'(x)}$. The letter has better
bound $r^{\frac{1}{3}+\e}$ than $g$. So by induction $|G|\les
r^{-(K+1)\e}$. We complete the proof of the lemma.
\end{proof}

We will not only use the decay of $I(\lambda)$, but also the
oscillation of $I(\lambda)$. Following the argument of section 3.4
in \cite{Zwor} (see also \cite{Stein2}), we prove

\begin{lem}\label{lem:Ilambda}
Assume $a$ is supported in $\{x\in \R: |x|<1\}$, and $\phi$ satisfies\\
(1) $\phi(0)=\phi'(0)=0$; \\
(2) $\phi''(x)\sim 1$, if $|x|\leq 1$; \\
(3) $|\phi^{(k)}(x)|\les 1$ if $|x|\leq 1$, $1\leq k\leq 6$.\\
Then
\[I(\lambda)=(2\pi)^{1/2}\lambda^{-1/2}\frac{a(0)e^{i\pi/4}}{[\phi''(0)]^{1/2}}+R(\lambda)\]
and for some $C$ independent of $\phi, a, \lambda$, we have
\[|R(\lambda)|\leq C\lambda^{-3/2}\sum_{0\leq k\leq 4}\sup_{\R}|\p^k a|.\]
Moreover, if assuming $|a^{(k)}(x)|\les 1$ and $|\phi^{(k)}(x)|\les
1$ for $|x|<1$, $k\in \N$, then for any $K\in \N$ we have
\[R(\lambda)=(2\pi)^{1/2}\lambda^{-1/2}e^{i\pi/4}\sum_{k=1}^K\frac{a_k\lambda^{-k}}{k!}+\tilde R(\lambda)\]
where $|a_k|\leq C^k$ and $|\tilde R(\lambda)|\leq
C_K\lambda^{-K-3/2}$ with constants $C,C_K$ independent of
$\phi,\lambda$.
\end{lem}

\begin{proof}
Since $\phi'(x)=x\cdot \int_0^1\phi''(tx)dt$, then we have
$\phi'(x)\sim x$ if $0<|x|<1$. We write
$\phi(x)=\frac{1}{2}\psi(x)x^2$ with
$\psi(x)=2\int_0^1(1-t)\phi''(tx)dt$, then we know $\psi(x)\sim 1$
if $|x|<1$. Make a change of variables
\[y:=\psi(x)^{1/2}x.\]
We see that $\p_y x=\frac{x}{\phi'(x)}\psi(x)^{1/2}\sim 1$ if
$|x|<1$ and $x_y(0)=\frac{1}{[\phi''(0)]^{1/2}}$. So it determines
unique a function $x=x(y)$. Moreover, using the equality
\[\p_yx=\frac{x}{\phi'(x)}\psi(x)^{1/2}=(\int_0^1\phi''(tx)dt)^{-1}(2\int_0^1(1-t)\phi''(tx)dt)^{1/2}\]
and the condition (3) we easily get
\begin{align}\label{eq:xybound}
|\p_y^{k}(x)|\les 1,\quad 1\leq k\leq 5.
\end{align}
Thus
\[I(\lambda)=\int e^{\frac{i\lambda y^2}{2}}a(x(y))x_y(y)dy.\]
Let $u(y)=a(x(y))x_y(y)$. Using the fact $\ft (e^{-\frac{i\lambda
y^2}{2}})=(2\pi
\lambda^{-1})^{1/2}e^{-\frac{i\pi}{4}}e^{\frac{i\xi^2}{2\lambda}}$,
we get
\begin{align*}
I(\lambda)=&(2\pi\lambda^{-1})^{1/2}e^{\frac{i\pi}{4}}\int
e^{-\frac{i\xi^2}{2\lambda}}\hat u(\xi)d\xi\\
=&(2\pi
\lambda^{-1})^{1/2}e^{\frac{i\pi}{4}}\frac{a(0)}{[\phi''(0)]^{1/2}}+\frac{(2\pi
\lambda^{-1})^{1/2}e^{\frac{i\pi}{4}}}{2\lambda}\int
\frac{e^{(-\frac{i\xi^2}{2\lambda}}-1)2\lambda}{-i\xi^2}(-i\xi^2)\hat
u(\xi)d\xi\\
:=&(2\pi\lambda^{-1})^{1/2}e^{\frac{i\pi}{4}}\frac{a(0)}{[\phi''(0)]^{1/2}}+R(\lambda).
\end{align*}
Now we estimate $R(\lambda)$. By \eqref{eq:xybound} we get
\begin{align*}
|R(\lambda)|\les \lambda^{-3/2}\norm{\widehat{\p^2u}}_{L^1}\les
\lambda^{-3/2}\sum_{0\leq k\leq 2}\sup_{\R}|\p^k \p^2 u|\les
\lambda^{-3/2}\sum_{0\leq k\leq 4}\sup_{\R}|\p^k a|.
\end{align*}

Moreover, if $|a^{(k)}(x)|\les 1$ and $|\phi^{(k)}(x)|\les 1$ for
$|x|<1$, $k\in \N$, then \eqref{eq:xybound} holds for $k\in \N$.
Then by the Taylor's expansion
$e^{-\frac{i\xi^2}{2\lambda}}=\sum_{k=0}^\infty
\frac{(-i\xi^2/(2\lambda))^k}{k!}$, we can prove the expansion for
$R(\lambda)$ with
\[a_k=2^{-k}i^k\p^{2k}u(0).\]
We complete the proof of the lemma.
\end{proof}

\begin{rem}
(1) If $\phi$ depends on $\lambda$, $a$ is independent of $\lambda$,
then $x=x(y,\lambda)$, and
\[\p_\lambda x=-\frac{\p_\lambda \phi}{\p_x\phi}, \quad \p_{y\lambda }^2x=-\frac{y\p_x^2\phi\p_\lambda \phi}{(\p_x\phi)^3}.\]
Then we get
\[|\p_\lambda R(\lambda)|\les \lambda^{-5/2}+\lambda^{-3/2}\sup_{|x|\les 1}\frac{|\p_\lambda \phi|}{|x|^2}.\]

(2) Lemma \ref{lem:Ilambda} applies easily to the general case. If
$\phi$ satisfies $\phi'(x_0)=0$ for some $x_0\in \supp (a)$ and
$\phi'(x)\ne 0$ for $x_0\ne x\in \supp (a)$, then under suitable
conditions
\[I(\lambda)=(2\pi\lambda^{-1})^{1/2}|\phi''(x_0)|^{-1/2}e^{\frac{i\pi}{4}\sgn \phi''(x_0)}e^{i\lambda \phi(x_0)}a(x_0)+O(\lambda^{-3/2}).\]
\end{rem}

We need to deal with $J_\nu(r)$ on the region $r>\nu+\nu^{1/3}$
which is usually the main contribution. The difficulty is that we
need to catch both decay and oscillation, especially in the
transitive region $\nu+\nu^{1/3}<r<2\nu$. In the case $d=2$, we need
a uniform expansion of the Bessel functions in this region. We prove

\begin{lem}[Asymptotical property]\label{lem:Bessel}
Let $\nu>10$ and $r>\nu+\nu^{1/3}$. Then

(1) We have
\[J_\nu(r)=\frac{1}{\sqrt{2\pi}}\frac{e^{i\theta(r)}+e^{-i\theta(r)}}{(r^2-\nu^2)^{1/4}}+h(\nu,r),\]
where
\[\theta(r)=(r^2-\nu^2)^{1/2}-\nu \arccos \frac{\nu}{r}-\frac \pi 4\]
and
\[|h(\nu,r)|\les \bigg(\frac{\nu^2}{(r^2-\nu^2)^{7/4}}+\frac{1}{r}\bigg)1_{[\nu+\nu^{1/3},2\nu]}(r)+r^{-1}1_{[2\nu,\infty)}(r).\]

(2) Let $x_0=\arccos \frac{\nu}{r}$. For any $K\in \N$ we have
\begin{align*}
h(\nu,r)=&(2\pi)^{-1/2}e^{i\theta(r)}x_0\sum_{k=1}^K\frac{
(rx_0^3)^{-k-1/2}a_k(x_0)}{k!}\\
&+(2\pi)^{-1/2}e^{-i\theta(r)}x_0\sum_{k=1}^K\frac{(rx_0^3)^{-k-1/2}\tilde
a_k(x_0)}{k!}+\tilde h(\nu,r)
\end{align*}
with functions $|\p^l a_k|+|\p^l \tilde a_k|\les 1$ for any $l\in
\N$ and
\[|\tilde h(\nu,r)|\les \bigg(\frac{r^{\frac{K}{2}+\frac{1}{4}}}{(r-\nu)^{\frac{3K}{2}+7/4}}+\frac{1}{r}\bigg)1_{[\nu+\nu^{1/3},2\nu]}(r)+r^{-1}1_{[2\nu,\infty)}(r).\]
Moreover, if $\nu\in \Z$, we have better estimate
\[|\tilde h(\nu,r)|\les \frac{r^{\frac{K}{2}+\frac{1}{4}}}{(r-\nu)^{\frac{3K}{2}+7/4}}1_{[\nu+\nu^{1/3},2\nu]}(r)+r^{-3/2}1_{[2\nu,\infty)}(r).\]

\end{lem}
\begin{proof}
Part (1) was given in \cite{BC}. Here we give a proof by Lemma
\ref{lem:Ilambda}. If $\nu\in \Z$, then $J_\nu^E(r)=0$. Thus it
suffices to consider $J_\nu^M(r)$. Denote $\phi(x)=\sin
x-\frac{\nu}{r}x$. Let $\phi'(x)=\cos x-\frac{\nu}{r}=0$, then we
find two solutions $x=\pm x_0=\pm \arccos \frac{\nu}{r}$. Since
$\nu<r$, we get $x_0\sim \frac{\sqrt{r^2-\nu^2}}{r}<1$. We divide
the proof into two cases.

{\bf Case 1.} $r\geq 2\nu$.

In this case we have $x_0\sim 1$. Let $\beta(x)$ be a cutoff
function around $0$ and supported in $\{|x|\ll 1\}$. Let $\tilde
\beta=1-\beta(x-x_0)-\beta(x+x_0)$. Then
\[J_\nu^M(r)=\frac{1}{2\pi}\int_{-\pi}^\pi e^{ir\phi(x)}[\beta(x-x_0)+\beta(x+x_0)+\tilde \beta(x)]dx:=I_1+I_2+I_3.\]
First, we estimate the term $I_3$. Since $|\phi'(x)|\sim 1$ in
$\supp \tilde\beta$, integrating by part we get that
\begin{align*}
|I_3|\les |\int_{-\pi}^\pi \frac{\p_x[e^{ir\phi(x)}]}{ir\phi'(x)}
\tilde\beta(x)dx|\les r^{-1}.
\end{align*}
If $\nu\in \Z$, we can do better since the boundary term vanishes.
Indeed, in this case from the fact that
$e^{ir\phi(\pi)}=e^{ir\phi(-\pi)}$, $\phi'(\pi)=\phi'(-\pi)$,
$\tilde\beta(\pi)=\tilde\beta(-\pi)$, we can get
\begin{align*}
|I_3|\les r^{-2}, \quad \mbox{ if } \nu\in \Z.
\end{align*}
Now we consider the term $I_1$. We have
\[I_1=\frac{1}{2\pi}\int e^{ir\phi(x+x_0)}\beta(x)dx.\]
It is easy to check that $\phi(x+x_0)-\phi(x_0),\beta$ satisfy the
conditions in Lemma \ref{lem:Ilambda}. Thus by Lemma
\ref{lem:Ilambda} we get
\[I_1=\frac{1}{\sqrt{2\pi}}\frac{e^{i\theta(r)}}{(r^2-\nu^2)^{1/4}}+R_1(\nu,r)\]
with $|R_1|\les r^{-3/2}$. Similarly, for $I_2$ we have
\[I_2=\frac{1}{\sqrt{2\pi}}\frac{-e^{i\theta(r)}}{(r^2-\nu^2)^{1/4}}+R_2(\nu,r)\]
with $|R_2|\les r^{-3/2}$. Therefore, we prove part (1) by setting
$h=R_1+R_2+I_3+J_\nu^E$.

{\bf Case 2.} $r<2\nu$.

Let $\gamma=1-\beta(\frac{x-x_0}{x_0})-\beta(\frac{x+x_0}{x_0})$.
Then
\[J_\nu^M(r)=\frac{1}{2\pi}\int_{-\pi}^\pi e^{ir\phi(x)}[\beta(\frac{x-x_0}{x_0})+\beta(\frac{x+x_0}{x_0})+\gamma(x)]dx:=II_1+II_2+II_3.\]
First, we estimate the term $I_1$. We have
\begin{align*}
II_1=\frac{1}{2\pi}\int
e^{ir\phi(x)}\beta(\frac{x-x_0}{x_0})dx=\frac{x_0}{2\pi}\int
e^{irx_0^3\cdot x_0^{-3}\phi(x_0x+x_0)}\beta(x)dx.
\end{align*}
By the condition $r>\nu+\nu^{1/3}$ we get $rx_0^3\ges 1$. Let
$\tilde\phi(x)=x_0^{-3}[\phi(x_0x+x_0)-\phi(x_0)]$. By the mean
value formula we can verify the conditions in Lemma \ref{} for
$\tilde\phi(x),\beta$. Thus by Lemma \ref{lem:Ilambda} we get
\[II_1=\frac{1}{\sqrt{2\pi}}\frac{e^{i\theta(r)}}{(r^2-\nu^2)^{1/4}}+\tilde R_1\]
with $|\tilde R_1|\les x_0(rx_0^3)^{-3/2}\les
\frac{\nu^2}{(r^2-\nu^2)^{7/4}}$. Similarly, for $II_2$ we get
\[II_2=\frac{1}{\sqrt{2\pi}}\frac{-e^{i\theta(r)}}{(r^2-\nu^2)^{1/4}}+\tilde R_2\]
with $|\tilde R_2|\les x_0(rx_0^3)^{-3/2}\les
\frac{\nu^2}{(r^2-\nu^2)^{7/4}}$.

Now we estimate the term $II_3$. We have
\[II_3=\frac{1}{2\pi}\int
e^{ir\phi(x)}\eta(x)\gamma(x)dx+\frac{1}{2\pi}\int_{-\pi}^{\pi}
e^{ir\phi(x)}(1-\eta(x))dx:=II_3^1+II_3^2.\] For the term $II_3^1$,
it's easy to see that $|\phi'(x)|\ges x_0^2,\,
|\p_x^k(\frac{1}{\phi'(x)})|\les x_0^{-2-k}$, $\forall k\in \N$ for
$x\in \supp (\eta\gamma)$, then integrating by parts we get that
\[|II_3^1|\les r^{-K}\bigg|\int
e^{ir\phi(x)}(\p_x
\frac{1}{\phi'(x)})^K\big[\eta(x)\gamma(x)\big]dx\bigg|\les x_0
(rx_0^3)^{-K}\les \frac{\nu^2}{(r^2-\nu^2)^{7/4}}.\] For the term
$II_3^2$, we have $|\phi'(x)|\sim 1$ for $x\in \supp(1-\eta)$. Thus
we get $|II_3^2|\les r^{-1}$ using integration by parts. If $\nu\in
\Z$, as in case 1, the boundary value vanishes, and we get
$|II_3^2|\les r^{-2}$. Thus we prove part (1) by setting $h=\tilde
R_1+\tilde R_2+II_3+J_\nu^E$.

Now we prove part (2). We only need to consider $\tilde R_1,\tilde
R_2$ in case 2. By Lemma \ref{lem:Ilambda} we have
\begin{align*}
\tilde
R_1=&x_0(2\pi)^{-1/2}e^{ir\phi(x_0)}(rx_0^3)^{-1/2}e^{i\pi/4}\sum_{k=1}^K\frac{a_k(rx_0^3)^{-k}}{k!}+x_0O((rx_0^3)^{-K-3/2})\\
=&(2\pi)^{-1/2}e^{i\theta(r)}x_0\sum_{k=1}^K\frac{a_k
(rx_0^3)^{-k-1/2}}{k!}+x_0O((rx_0^3)^{-K-3/2}).
\end{align*}
We can obtain the expansion for $\tilde R_2$ similarly. We complete
the proof.
\end{proof}

\section{Spherically averaged Strichartz estimates}

In this section, we prove Theorem \ref{thm1} by improving the proof
in \cite{GLNW}. First, we reproduce some proof in \cite{GLNW} for
the readers' convenience. To prove \eqref{eq:striestwea}, it is
equivalent to show
\begin{align}\label{eq:Stri2}
\norm{T_af}_{L_t^q\Lr_{\rho}^pL_\omega^2}\les \norm{f}_{L_x^2},
\end{align}
where
\[T_af(t,x)=\int_{\R^d}e^{i(x\xi+t|\xi|^a)}\chi_0(|\xi|)f(\xi)d\xi.\]
Now we expand $f$ by the orthonormal basis $\{Y_k^l\}$, $k\geq
0,1\leq l\leq d(k)$ of spherical harmonics with
$d(k)=C_{n+k-1}^k-C_{n+k-3}^{k-2}$, such that
\[f(\xi)=f(\rho \sigma)=\sum_{k\geq 0}\sum_{1\leq l\leq d(k)}a_k^l(\rho)Y_k^l(\sigma).\]
Using the identities (see \cite{Stein1})
\[\widehat{Y_k^l}(\rho\sigma)=c_{d,k}\rho^{-\frac{d-2}{2}}J_\nu(\rho)Y_k^l(\sigma)\]
where $c_{d,k}=(2\pi)^{d/2}i^{-k}$, $\nu=\nu(k)=\frac{d-2+2k}{2}$,
then we get
\[T_af(t,x)=\sum_{k,l}c_{d,k}T_a^\nu (a_k^l)(t,|x|)Y_k^l(x/|x|),\]
where
\[T_a^\nu(h)(t,r)=r^{-\frac{d-2}{2}}\int e^{-it\rho^a}J_\nu(r\rho)\rho^{d/2}\chi_0(\rho)h(\rho)d\rho.\]
Here $J_\nu(r)$ is the Bessel function. Thus \eqref{eq:Stri2}
becomes
\begin{align}\label{eq:Stri3}
\norm{T_a^\nu (a_k^l)}_{L_t^q\Lr_{r}^pl_{k,l}^2}\les
\norm{\{a_k^l(\rho)\}}_{\Lr_\rho^2l_{k,l}^2}.
\end{align}
To prove \eqref{eq:Stri3}, it is equivalent to show
\begin{align}\label{eq:goal}
\norm{T_a^\nu (h)}_{L_t^q\Lr_{r}^p}\les \norm{h}_{L^2},
\end{align}
with a bound independent of $\nu$, since $q,p\geq 2$.

By the classical Strichartz estimates (see the endpoint estimates in
\cite{KT, MaNaNaOz05}
), we can get $\norm{1_{r\leq 100}T_a^\nu (h)}_{L_t^q\Lr_{r}^p}\les
\norm{h}_{L^2}$. Thus it remains to show
\begin{align}\label{eq:goal2}
\norm{1_{r\gg 1}T_a^\nu (h)}_{L_t^q\Lr_{r}^p}\les \norm{h}_{L^2},
\end{align}
with a bound independent of $\nu$. For any $R\gg 1$, define
\[S_R^{\nu,a}(h)(t,r)=\chi_0\big(\frac rR\big)\int e^{-it\rho^a}J_\nu(r\rho)\chi_0(\rho)h(\rho)d\rho.\]
Then
\[\norm{1_{r\gg 1}T_a^\nu}_{L^2\to L_t^q\Lr_{r}^p}\les \sum_{j\geq 5}2^{j(\frac{d-1}{p}-\frac{d-2}{2})}\norm{S_{2^j}^{\nu,a}}_{L^2\to L_t^qL_{r}^p}.\]
Then to prove \eqref{eq:goal2}, it suffices to show for some
$\delta>0$
\begin{align}\label{eq:goal3}
R^{\frac{d-1}{p}-\frac{d-2}{2}}\norm{S_R^{\nu,a}
(h)}_{L_t^qL_{r}^p}\leq C R^{-\delta}\norm{h}_{L^2},
\end{align}
where $C$ is independent of $\nu$. By interpolation, we only need to
show \eqref{eq:goal3} for $(q,p)=(2,p)$.  The difficulty in
\eqref{eq:goal3} is to obtain a uniform bound as $\nu\to \infty$. We
need to exploit the uniform properties of the Bessel function with
respect to $\nu$. By the uniform decay of Bessel function presented
in Lemma \ref{lem:Bessdecay}, one can show

\begin{lem}[Lemma 2.4, \cite{GLNW}]\label{prop:roughes}
Assume $a>0$. For $\nu>10$, $R\ges 1$, $2\leq p\leq \infty$
\begin{align}
\norm{S_R^{\nu,a} (h)}_{L_t^2 L_{r}^p}\les \norm{h}_{L^2}, \quad
\norm{S_R^{\nu,a} (h)}_{L_t^\infty L_{r}^2}\les \norm{h}_{L^2}.
\end{align}
\end{lem}

This lemma gives the sharp estimates for the wave equation $a=1$.
For the Schr\"odinger case $a>1$, there is some more oscillatory
effect to exploit. To do so, in \cite{GLNW}, $S_R^{\nu,a}$ is
decomposed into three operators $S_R^\nu (h)=\sum_{j=1}^3S_{R,j}^\nu
(h)$
 where
\[S_{R,j}^\nu (h)=\chi_0\big(\frac rR\big)\int
e^{-it\rho^2}J_\nu(r\rho)\gamma_j(\frac{r\rho-\nu}{\lambda})\chi_0(\rho)h(\rho)d\rho,\]
with $\gamma_1(x)=\eta(x)$, $\gamma_2(x)=(1-\eta(x))1_{x<0}$, and
$\gamma_3(x)=(1-\eta(x))1_{x>0}$. By some uniform stationary phase
analysis, the following lemma was proved in \cite{GLNW}

\begin{lem}[Lemma 2.5, \cite{GLNW}]\label{lem:SR} Assume $R\ges 1, \lambda\geq 100 R^{1/3}$, $2\leq p\leq \infty$. Then
\begin{align*}
\norm{S_{R,1}^\nu (h)}_{L_t^2 L_{r}^p}\les&
\lambda^{1/4}R^{-1/4}\norm{h}_{L^2},\\
\norm{S_{R,2}^\nu (h)}_{L_t^2
L_{r}^p}\les& \big((\lambda^{-1}R^{1/4})^{1-\frac{2}{p}}+R^{-1/2}\big)\norm{h}_{L^2},\\
\norm{S_{R,3}^\nu (h)}_{L_t^2 L_{r}^p}\les&
\big(\lambda^{-\frac{1}{4}(1-\frac{2}{p})}+(\lambda^{-5/4}R^{1/4})^{2/p}+R^{-1/p}\big)\norm{h}_{L^2}.
\end{align*}
\end{lem}

\subsection{The improvement: non-endpoint}
To prove Theorem \ref{thm1}, we will refine the estimates for
$S_{R,3}^\nu$ and $S_{R,2}^\nu$. We prove

\begin{lem}\label{lem:SR3imp} Assume $R\ges 1, \lambda\geq 100 R^{1/3}$, $2\leq p\leq \infty$.
Then for any $K\in \N$
\begin{align}
\norm{S_{R,3}^\nu (h)}_{L_t^2 L_{r}^p}\les&
\big(R^{-\frac{1}{4}(1-\frac{2}{p})}+(R^{-1/2}\lambda^{3/2})^{-2K/p}(\lambda^{-5/4}R^{1/4})^{2/p}+R^{-1/p}\big)\norm{h}_{L^2}.
\end{align}
Moreover, if $\nu\in \Z$ we have
\begin{align}
\norm{S_{R,3}^\nu (h)}_{L_t^2 L_{r}^p}\les&
\big(R^{-\frac{1}{4}(1-\frac{2}{p})}+(R^{-1/2}\lambda^{3/2})^{-2K/p}(\lambda^{-5/4}R^{1/4})^{2/p}+R^{-2/p}\big)\norm{h}_{L^2}.
\end{align}
\end{lem}

\begin{proof}
By interpolation, we only need to show the estimates for $p=2,
\infty$. By the support of $\gamma_3$, we have
$r\rho>\nu+\lambda>\nu+\nu^{1/3}$ in the support of
$\gamma_3(\frac{r\rho-\nu}{\lambda})$. Thus we use the Lemma
\ref{lem:Bessel}, and decompose
\[S_{R,3}^\nu(h):=M_{R,3}^\nu(h)+E_{R,3}^\nu(h)\]
where
\begin{align*}
M_{R,3}^\nu (h)=&\chi_0\big(\frac rR\big)\int
e^{-it\rho^a}\frac{e^{i\theta(r\rho)}+e^{-i\theta(r\rho)}}
{2 \sqrt{2\pi}(r^2\rho^2-\nu^2)^{1/4}}\gamma_3(\frac{r\rho-\nu}{\lambda})\chi_0(\rho)h(\rho)d\rho,\\
E_{R,3}^\nu (h)=&\chi_0\big(\frac rR\big)\int
e^{-it\rho^a}h(\nu,r\rho)\gamma_3(\frac{r\rho-\nu}{\lambda})\chi_0(\rho)h(\rho)d\rho,
\end{align*}
with $\theta(r),h(\nu,r)$ given in Lemma \ref{lem:Bessel}.

{\bf Step 1.} The estimate for $M_{R,3}^\nu$.

We only estimate
\[\widetilde{M}_{R,3}^\nu (h)=\chi_0\big(\frac rR\big)\int e^{-it\rho^a}\frac{e^{i\theta(r\rho)}}
{(r^2\rho^2-\nu^2)^{1/4}}\gamma_3(\frac{r\rho-\nu}{\lambda})\chi_0(\rho)h(\rho)d\rho,\]
since the other term is similar. It is easy to see
$\norm{\widetilde{M}_{R,3}^\nu (h)}_{L_t^2L_r^2}\les \norm{h}_2$ by
Plancherel's equality in $t$. It suffices to
\[\norm{\widetilde{M}_{R,3}^\nu(h)}_{L_t^2L_r^\infty}\les R^{-1/4}\norm{h}_{L^2}.\]
Let $\gamma(x)=\chi(x)\cdot 1_{x>0}$. We decompose further
$\widetilde{M}_{R,3}^\nu(h)=\sum_{k: R^{1/3}\ll 2^k \les
R}\widetilde{M}^\nu_{R,3,k}(h)$ where
\begin{align}\label{eq:MR3k}
\widetilde{M}^\nu_{R,3,k}(h)=\chi_0\big(\frac rR\big)\int
e^{-it\rho^a}J_\nu(r\rho)\gamma(\frac{r\rho-\nu}{2^k})\chi_0(\rho)h(\rho)d\rho.
\end{align}
It suffices to prove
\begin{align}
\norm{\widetilde{M}^\nu_{R,3,k}(h)}_{L_t^2L_r^\infty}\les
2^{k/8}R^{-3/8}\norm{h}_2.
\end{align}
By $TT^*$ argument, it suffices to prove
\begin{align}
\norm{\widetilde{M}_{R,3,k}^\nu(\widetilde{M}_{R,3,k}^\nu)^*(f)}_{L_t^2L_r^\infty}\les
2^{k/4}R^{-3/4}\norm{f}_{L_t^2L_r^1}.
\end{align}
The kernel for
$\widetilde{M}_{R,3,k}^\nu(\widetilde{M}_{R,3,k}^\nu)^*$ is
\[K(t-t',r,r')=\int e^{-i[(t-t')\rho^a-\theta(r\rho)+\theta(r'\rho)]}\frac{\chi_0\big(\frac rR\big)\gamma(\frac{r\rho-\nu}{2^k})}
{(r^2\rho^2-\nu^2)^{1/4}}\frac{\chi_0\big(\frac
{r'}R\big)\gamma(\frac{r'\rho-\nu}{2^k})}
{(r'^2\rho^2-\nu^2)^{1/4}}\chi_0^2(\rho)d\rho.\] Obviously, we have
a trivial bound
\[|K|\les 2^{-k/2}R^{-1/2}.\]
Recall
$\theta(r)=(r^2-\nu^2)^{1/2}-\nu\arccos\frac{\nu}{r}-\frac{\pi}{4}$,
then direct computation shows
\begin{align*}
\theta'(r)=&(r^2-\nu^2)^{1/2}r^{-1},\\
\theta''(r)=&(r^2-\nu^2)^{-1/2}-(r^2-\nu^2)^{1/2}r^{-2}=(r^2-\nu^2)^{-1/2}\nu^2r^{-2},\\
\theta'''(r)=&(r^2-\nu^2)^{-3/2}\frac{\nu^2}{r}(-3+\frac{2\nu^2}{r^2}).
\end{align*}
Denoting $G=\frac{\chi_0\big(\frac
rR\big)\gamma_3(\frac{r\rho-\nu}{\lambda})}
{(r^2\rho^2-\nu^2)^{1/4}}\frac{\chi_0\big(\frac
{r'}R\big)\gamma_3(\frac{r'\rho-\nu}{\lambda})}
{(r'^2\rho^2-\nu^2)^{1/4}}\chi_0^2(\rho)$,
$\phi_2=t\rho^a-\theta(r\rho)+\theta(r'\rho)$. Then
\begin{align*}
\partial_\rho(\phi_2)=&at\rho^{a-1}-\theta'(r\rho)r+\theta'(r'\rho)r'=2t\rho-
\frac{\rho(r'^2-r^2)}{\sqrt{r'^2\rho^2-\nu^2}+\sqrt{r^2\rho^2-\nu^2}}\\
\partial^2_\rho(\phi_2)=&a(a-1)t\rho^{a-2}-\theta''(r\rho)r^2+\theta''(r'\rho)r'^2\\
=&a(a-1)t\rho^{a-2}+
\frac{(r'^2-r^2)}{\sqrt{r'^2\rho^2-\nu^2}+\sqrt{r^2\rho^2-\nu^2}}
\frac{\nu^2}{\sqrt{r'^2\rho^2-\nu^2}\sqrt{r^2\rho^2-\nu^2}}\\
\partial^3_\rho(\phi_2)=&-\theta'''(r\rho)r^3+\theta'''(r'\rho)r'^3.
\end{align*}
We divide the proof into two cases.

{\bf Case 1.} $R\les \nu$

The key observation here is that if $|\partial_\rho(\phi_2)|\ll
|t|$, then $|\partial^2_\rho(\phi_2)|\ges |t|R2^{-k}$ on the support
of $G$. Indeed, if $|\partial_\rho(\phi_2)|\ll |t|$, then $t$ and
$\frac{r'^2-r^2}{\sqrt{r'^2\rho^2-\nu^2}+\sqrt{r^2\rho^2-\nu^2}}$
have same sign and comparable size. This observation is not true if
$a<1$, in which case $\p_\rho(\phi_2)$ and $\p_\rho(\phi_2)$ can
both be small. Note that on the support of $G$, one has
\[\frac{|r'^2-r^2|}{\sqrt{r'^2\rho^2-\nu^2}+\sqrt{r^2\rho^2-\nu^2}}\les
2^{k/2}R^{1/2}.\] If $|t|\les 2^{k/2}R^{1/2}$, we divide $K$
\[K=\int e^{-i\phi_2}G \eta_0(\frac{100\partial_\rho(\phi_2)}{t})d\rho+
\int e^{-i\phi_2}G
[1-\eta_0(\frac{100\partial_\rho(\phi_2)}{t})]d\rho:=I_1+I_2.\] By
Lemma \ref{lem:staph}, we obtain
\begin{align*}
|I_1|\les& |t|^{-1/2}\bigg(\int |\partial_\rho G
\eta_0(\frac{100\partial_\rho(\phi_2)}{t})|d\rho+\int |G
\eta_0'(\frac{100\partial_\rho(\phi_2)}{t})\frac{100\partial_\rho^2(\phi_2)}{t}|d\rho\bigg)\\
\les& |t|^{-1/2}2^{k/2}R^{-1/2}2^{-k/2}R^{-1/2},
\end{align*}
where for the first term, we estimate it as $K_3$, while for the
second term, we only need to observe that
$\eta_0'(\frac{100\partial_\rho(\phi_2)}{t})\frac{100\partial_\rho^2(\phi_2)}{t}$
has fixed sign depending only on $t$.

For $I_2$, without loss of generality, we assume $r^2-r'^2>0$. Then
integrating by part, we get
\begin{align*}
|I_2|\les& \int \aabs{\partial_\rho \bigg((\partial_\rho\phi_2)^{-1}G[1-\eta_0(\frac{100\partial_\rho(\phi_2)}{t})]\bigg)}d\rho\\
\les& \int
\aabs{\frac{\partial^2_\rho\phi_2}{(\partial_\rho\phi_2)^2}G[1-\eta_0(\frac{100\partial_\rho(\phi_2)}{t})]}d\rho\\
&+|t|^{-1}(\int |\partial_\rho G|d\rho+\int |G
\eta_0'(\frac{100\partial_\rho(\phi_2)}{t})\frac{100\partial_\rho^2(\phi_2)}{t}|d\rho)\\
\les&\int
\frac{|\partial^2_\rho\phi_2|}{(\partial_\rho\phi_2)^2}G[1-\eta_0(\frac{100\partial_\rho(\phi_2)}{t})]d\rho+
|t|^{-1}\lambda^{-1/2}R^{-1/2}\\
\les&|t|^{-1}2^{-k/2}R^{-1/2},
\end{align*}
where we used the fact that $\partial_\rho^2 \phi_2$ changes the
sign at most once.

If $|t|\gg 2^{k/2}R^{1/2}$, we have $|\partial_\rho(\phi_2)|\sim
|t|$. Thus integrating by part, we get
\begin{align*}
|K|\les& \int \aabs{\partial_\rho \big[(\partial_\rho\phi_2)^{-1}\partial_\rho \big((\partial_\rho\phi_2)^{-1}G\big)\big]}d\rho\\
\les& \int
\aabs{(\partial_\rho\phi_2)^{-3}\partial^3_\rho\phi_2G}d\rho+\int
\aabs{(\partial_\rho\phi_2)^{-2}\partial_\rho^2G}d\rho\\
&+\int
\aabs{(\partial_\rho\phi_2)^{-3}\partial^2_\rho\phi_2\partial_\rho
G}d\rho+\int
\aabs{(\partial_\rho\phi_2)^{-4}(\partial^2_\rho\phi_2)^2 G}d\rho\\
:=&II_1+II_2+II_3+II_4.
\end{align*}
As for $I_2$, we can obtain
\[II_2+II_3+II_4\les |t|^{-2}2^{-k/2}R^{-1/2}R^22^{-2k}\les |t|^{-2}2^{-5k/2}R^{3/2}.\]
For $II_1$, we have
\begin{align*}
II_1\les& |t|^{-3}\lambda^{-1/2}R^{-1/2}\int
(-\theta'''(r\rho)r^3-\theta'''(r'\rho)r'^3)
\gamma_3(\frac{r\rho-\nu}{\lambda})\gamma_3(\frac{r'\rho-\nu}{\lambda})d\rho\\
\les& |t|^{-3}\lambda^{-1/2}R^{-1/2}
\sup_{\rho:r\rho>\nu+\lambda}\theta''(r\rho)r^2\les |t|^{-3}2^{-k}R.
\end{align*}
Thus, eventually we get
\[|K|\les |t|^{-1/2}R^{-1}1_{|t|\les R^{1/2}2^{k/2}}+(|t|^{-2}2^{-5k/2}R^{3/2}+|t|^{-3}2^{-k}R)1_{|t|\gg R^{1/2}2^{k/2}}\]
which implies $\norm{K}_{L_t^1L_{r,r'}^\infty}\les 2^{k/4}R^{-3/4}$
as desired, if $2^k\ges R^{1/3}$.

{\bf Case 2.} $R\gg \nu$.

In this case, we may assume $2^k\sim R$. We also observe that if
$|\partial_\rho(\phi_2)|\ll |t|$, then
$|\partial^2_\rho(\phi_2)|\ges |t|$ on the support of $G$. The rest
of proof is the same as Case 1.

{\bf Step 2.} The estimate for $E_{R,3}^\nu$.

First, we have for any $f\in L^2$
\[\norm{E_{R,3}^\nu(f)}_{L_t^2L_r^\infty}\les \norm{S_{R,3}^\nu(f)}_{L_t^2L_r^\infty}+\norm{M_{R,3}^\nu(f)}_{L_t^2L_r^\infty}\les \norm{f}_{L^2}.\]
On the other hand, using the decay estimate of $h(\nu,r)$, we get
\begin{align}\label{eq:ER3}
\norm{E_{R,3}^\nu(f)}_{L_t^2L_r^2}\les
(\lambda^{-5/4}R^{1/4}+R^{-1/2})\norm{f}_{L^2}.
\end{align}
In the case $d=2$, $\nu\in \Z$, thus by the better decay of
$h(\nu,r)$ given by \cite{}, we can get
\begin{align}\label{eq:ER32d}
\norm{E_{R,3}^\nu(f)}_{L_t^2L_r^2}\les
(\lambda^{-5/4}R^{1/4}+R^{-1})\norm{f}_{L^2}.
\end{align}
Thus, the lemma with $K=0$ is proved by interpolation.

To show the case $K\geq 1$, we need to analyze $E_{R,3}^\nu$ more
carefully. Using the expansion in Lemma \ref{lem:Bessel}, we can
divide
\[E_{R,3}^\nu(f)=\sum_{k=1}^K E_{R,3,k}^\nu(f)+\tilde E_{R,3,K}\]
where $\tilde E_{R,3,k}$ is the term $E_{R,3,k}$ with $h$ replaced
by $\tilde h$. Arguing as before, we see $E^\nu_{R,3,k}$ has the
same bound as $M_{R,3}^\nu$. For $\tilde E_{R,3,K}$, we can obtain
the bound similarly as $E_{R,3}^\nu$. We complete the proof
\end{proof}

For the case $d=2$, we also need to refine the estimate for
$S_{R,2}^\nu$. By Lemma \ref{lem:Bessdecay2} and Sobolev embedding
we get

\begin{lem}\label{lem:SR2imp} Let $\nu\in \Z$, $R\ges 1$, $2\leq p\leq \infty$. If $\lambda\geq 100
R^{\frac{1}{3}+\e}$ for some $\e>0$, then for any $N>0$ there exists
$C_{N,\e}$ such that
\begin{align}
\norm{S_{R,2}^\nu (h)}_{L_t^2 L_{r}^p}\leq&
C_{N,\e}R^{-N\e}\norm{h}_{L^2}.
\end{align}
\end{lem}

Now we are ready to prove Theorem \ref{thm1} for the non-endpoint
range, namely assuming
\begin{align} 2\leq q,p\leq
\infty,\,\frac{1}{q}< (d-\frac{1}{2})(\frac{1}{2}-\frac{1}{p})\mbox{
or } (q,p)=(\infty,2).
\end{align}
It suffices to show \eqref{eq:goal3}. First, we consider $d\geq 3$.
From Lemma \ref{lem:SR} and Lemma \ref{lem:SR3imp} by taking
$\lambda=R^{\frac{1}{2}+}$ and $K=0$, we get
\[R^{\frac{d-1}{p}-\frac{d-2}{2}}\norm{S_R^{\nu,a}
(h)}_{L_t^2L_{r}^p}\les
R^{\frac{d-1}{p}-\frac{d-2}{2}}R^{-\frac{1}{4}(1-\frac{2}{p})}\norm{h}_2\les
R^{-\delta}\norm{h}_2\] for some $\delta>0$ if
$\frac{4d-2}{2d-3}<p<\frac{2d}{d-2}$. For $d=2$, we can prove
Theorem \ref{thm1} similarly by using Lemma \ref{lem:SR}, Lemma
\ref{lem:SR3imp} and Lemma \ref{lem:SR2imp} by taking
$\lambda=R^{1/3+\e}$ for $\e>0$ sufficiently small and $K$
sufficiently large.

\subsection{The improvement: endpoint} It remains to prove Theorem \ref{thm1} for $(q,p)$ lying on the
endpoint line for $d\geq 3$, namely $\frac{1}{q}=
(d-\frac{1}{2})(\frac{1}{2}-\frac{1}{p})$. From the proof in the
previous subsection, assuming $\lambda>R^{\frac{1}{2}+}$ we know the
logarithmic difficulty only appears in the summation $\sum_R\wt
M_{R,3}^\nu(h)$ (This is not true for $d=2$). We will exploit some
orthogonality to overcome this logarithmic difficulty. Similar
technique was also used in \cite{CL} and \cite{Ke}. It suffices to
show
\begin{align}\label{eq:endsum}
\norm{\sum_R r^{\frac{d-1}{p}-\frac{d-2}{2}}\wt
M_{R,3}^\nu(h)}_{L_t^qL_r^p}\leq \norm{h}_2
\end{align}
with a uniform bound with respect to $\nu$. By $TT^*$ argument, we
see that \eqref{eq:endsum} is equivalent to
\begin{align}\label{eq:endsum2}
\norm{\sum_{R',R} r^{\frac{d-1}{p}-\frac{d-2}{2}}\wt M_{R,3}^\nu
(\wt
M_{R',3}^\nu)^*(r'^{\frac{d-1}{p}-\frac{d-2}{2}}g)}_{L_t^qL_r^p}\leq
\norm{g}_{L_t^{q'}L_r^{p'}}.
\end{align}
The key ingredient to prove \eqref{eq:endsum2} is the following
observation:

\begin{lem}\label{lem:ortho}
Assume $d\geq 2$, $R\gg R'>1$ and $(q,p)$ satisfies $2\leq q,p\leq
\infty,\,\frac{1}{q}=(d-\frac{1}{2})(\frac{1}{2}-\frac{1}{p}),\,
(q,p)\ne (2,\frac{4d-2}{2d-3})$. Then $\exists \e>0$ such that
\begin{align}
\norm{ r^{\frac{d-1}{p}-\frac{d-2}{2}}\wt M_{R,3}^\nu (\wt
M_{R',3}^\nu)^*(r'^{\frac{d-1}{p}-\frac{d-2}{2}}g)}_{L_t^qL_r^p}\les
(R'/R)^\e \norm{g}_{L_t^{q'}L_r^{p'}}.
\end{align}
\end{lem}
\begin{proof}
By interpolation, it suffices to show Lemma \ref{lem:ortho} for
$q=p=q_0=\frac{4d+2}{2d-1}$. We may assume $R\gg R'\ges \nu$. Then
we decompose
\[M_{R,3}^\nu (\wt M_{R',3}^\nu)^*g=\sum_{k: R'^{\frac{1}{2}+}<2^k\les R'}M_{R,3}^\nu (\wt M_{R',3,k}^\nu)^*g\]
where $\wt M_{R,3,k}^\nu$ is given by \eqref{eq:MR3k}. We can write
\[\wt M_{R,3}^\nu (\wt M_{R',3,k}^\nu)^*g=\int \tilde K(t-t',r,r')g(t',r')dt'dr'\]
where
\[\tilde K(t-t',r,r')=\int e^{-i[(t-t')\rho^a-\theta(r\rho)+\theta(r'\rho)]}\frac{\chi_0\big(\frac rR\big)}
{(r^2\rho^2-\nu^2)^{1/4}}\frac{\chi_0\big(\frac
{r'}{R'}\big)\gamma(\frac{r'\rho-\nu}{2^k})}
{(r'^2\rho^2-\nu^2)^{1/4}}\chi_0^2(\rho)d\rho.\] By the stationary
phase method as for $K$ in the proof of Lemma \ref{lem:SR3imp}, we
obtain
\[|\tilde K|\les (R \frac{R'^{3/2}}{R 2^{k/2}})^{-1/2}R^{-1/2}R'^{-1/4}2^{-k/4}\les R^{-1}R'^{-1/2}.\]
Thus we get
\[\norm{r^{-\frac{d-2}{2}}\wt M_{R,3}^\nu (\wt M_{R',3,k}^\nu)^*(r^{-\frac{d-2}{2}}g)}_{L_t^\infty L_r^\infty}\les (RR')^{-\frac{d-2}{2}}R^{-1}R'^{-1/2}\norm{g}_{L_t^1 L_r^1}.\]
Interpolating with the following bound
\begin{align*}
\norm{r^{\frac{1}{2}}\wt M_{R,3}^\nu (\wt
M_{R',3,k}^\nu)^*(r'^{\frac{1}{2}}g)}_{L_t^2 L_r^2} \les&
R^{\frac{1}{2}}\norm{(\wt
M_{R',3,k}^\nu)^*(r'^{\frac{1}{2}}g)}_{L^2}\\
\les& R^{\frac{1}{2}}R'^{\frac{1}{2}}2^{k/4}R'^{-1/4}\norm{g}_{L_t^2
L_r^2}
\end{align*}
we obtain
\begin{align*}
&\norm{r^{\frac{d-1}{q_0}}\wt M_{R,3}^\nu (\wt
M_{R',3,k}^\nu)^*(r'^{\frac{d-1}{q_0}}g)}_{L_t^{q_0}L_r^{q_0}}\\
\les&
((RR')^{-\frac{d-2}{2}}R^{-1}R'^{-1/2})^{1-2/q_0}(R^{\frac{1}{2}}R'^{\frac{1}{2}}2^{k/4}R'^{-1/4})^{2/q_0}\norm{g}_{L_t^{q_0'}L_r^{q_0'}}\\
\les&
(R'/R)^{\frac{1}{4d+2}}(2^{k/4}R'^{-1/4})^{\frac{2d-1}{2d+1}}\norm{g}_{L_t^{q_0'}L_r^{q_0'}}.
\end{align*}
Therefore, summing over $k$ we complete the proof of Lemma
\ref{lem:ortho}.
\end{proof}

Now we prove \eqref{eq:endsum2}. We have
\begin{align*}
&\norm{\sum_{R,R'}r^{\frac{d-1}{p}-\frac{d-2}{2}}\wt M_{R,3}^\nu
(\wt
M_{R',3}^\nu)^*(r'^{\frac{d-1}{p}-\frac{d-2}{2}}g)}_{L_t^qL_r^p}\\
\les& (\sum_{R}\norm{\sum_{R'} r^{\frac{d-1}{p}-\frac{d-2}{2}}\wt
M_{R,3}^\nu (\wt
M_{R',3}^\nu)^*(r'^{\frac{d-1}{p}-\frac{d-2}{2}}\chi(r/R')g)}_{L_t^qL_r^p}^2)^{1/2}\\
\les& (\sum_{R}\norm{\chi(r/R)g}_{L_t^{q'}L_r^{p'}}^2)^{1/2}\les
\norm{g}_{L_t^{q'}L_r^{p'}}.
\end{align*}
where we used Lemma \eqref{lem:ortho} in the second inequality.

Interpolating Theorem \ref{thm1} with the classical Strichartz
estimates we get

\begin{cor}\label{cor:str}
Assume $a>1, d\geq 2$, $\frac{2d}{d-2}>p>\frac{4d-2}{2d-3}$. Let
$\beta(p)=\frac{2p(d-1)}{(4-p)d+2p-2}$. Then
\begin{align*}
\normo{e^{itD^a}P_0f}_{L_t^2\Lr_{\rho}^pL_\omega^{\beta(p)-}}\les&
\norm{f}_{L^2},
\end{align*}
where $a-$ denotes $a-\e$ for any $\e>0$.
\end{cor}

\subsection{Counter-example}
Finally, we use the Knapp example to obtain some necessary
conditions for the Strichartz estimates with mixed angular-radius
integrability, namely
\begin{align}\label{eq:Fschg}
\norm{e^{itD^a}P_0f}_{L_t^q\Lr_{\rho}^pL_\omega^s}\les
\norm{f}_{L_x^2}.
\end{align}
\begin{prop}
Assume $1\ne a>0$ and \eqref{eq:Fschg} holds. Then $\frac{2}{q}\leq
\frac{d}{2}-\frac{2d-1}{p}+\frac{d-1}{s}$. As a consequence,
$\beta(p)$ in Corollary \ref{cor:str} is sharp.
\end{prop}
\begin{proof}
Take
\[D=\{\xi=(\xi_1,\xi')\in \R^d:|\xi_1-1|\les \delta,|\xi'|\leq \delta\}.\]
Let $\hat{f}=1_D(\xi)$. Then $\norm{f}_2\sim \delta^{d/2}$, and
\[\int_{\R^d}e^{it|\xi|^\sigma}e^{ix\xi}\eta_0(\xi)\hat{f}(\xi)d\xi=e^{i(t+x_1)}\int_{D}
e^{it(|\xi|^a-\xi_1^a)}e^{it(\xi_1^a-1-a(\xi_1-1))}e^{i(ta+x_1)(\xi_1-1)}e^{ix'\xi'}d\xi.\]
Since in $D$ we have
\[||\xi|^a-\xi_1^a|\les |\xi'|^2\les \delta^2, \, |\xi_1^a-1-a(\xi_1-1)|\les |\xi_1-1|^2\les \delta^2,\]
then for $(t,x)\in E=\{|t|\les \delta^{-2},\, |ta+x_1|\les
\delta^{-1},\, |x'|\les \delta^{-1}\}$, we have
\[\bigg|\int_{\R^d}e^{it|\xi|^a}e^{ix\xi}\eta_0(\xi)\hat{f}(\xi)d\xi\bigg|\sim
|D|\sim \delta^d.\] By simple geometric observation we see that
\[E\supseteq E'= \{|t|\sim \delta^{-2}, \rho\in
(\delta^{-2},\delta^{-2}+\delta^{-1}), |\theta|<\delta\}\] where
$\theta$ is the central angle. Therefore, \eqref{eq:Fschg} implies
\[\delta^d\delta^{(d-1)/s}(\delta^{-2})^{\frac{d-1}{p}}\delta^{-1/p}\delta^{-\frac{2}{q}}\les \delta^{d/2},\]
which implies immediately that $\frac{2}{q}\leq
\frac{d}{2}-\frac{2d-1}{p}+\frac{d-1}{s}$ by taking $\delta\ll 1$.
\end{proof}

\section{3D Zakharov system}

This section is devoted to proving Theorem \ref{thm2}. We follow the
ideas in \cite{GN} and \cite{GLNW}. The new difficulty is to handle
the fractional derivatives on the sphere in accordance with Fourier
multiplier, we
 will transfer it to $SO(3)$. After normal form reduction (see \cite{GN}), the Zakharov system \eqref{eq:Zak} is equivalent to \EQ{
 \pt(i\p_t+D^2)(u-\Om(N,u))=(Nu)_{LH+HH+\al L}+\Om(D|u|^2,u)+\Om(N,Nu),
 \pr(i\p_t+\al D)(N-D\tilde\Om(u,u))=D|u|^2_{HH+\al L+L\al}+D\tilde \Om(Nu,u)+D\tilde\Om(u,Nu).}
 where $N=n - iD^{-1}\dot n/\al$, $\Om, \tilde \Om$ are bilinear Fourier multiplier operators with symbols $\omega,\tilde \omega$ respectively, where
\begin{align*}
\omega(\xi,\eta)=&\frac{\sum_{k\in \Z: |k-\log_2\alpha|>
1}\chi_k(\xi)\chi_{\leq k-5}(\eta)}{-|\x+\y|^2+\al|\x|+|\y|^2},\\
\omega(\xi,\eta)=&\frac{\sum_{k\in \Z: |k-\log_2\alpha|>
1}(\chi_k(\xi)\chi_{\leq k-5}(\eta)+\chi_k(\eta)\chi_{\leq
k-5}(\xi))}{|\xi|^2-|\eta|^2-\alpha|\xi+\eta|}.
\end{align*}
Here the bilinear Fourier multiplier operator with a symbol $m$ on
$\R^{6}$ is the bilinear operator $T_m$ defined by
$T_m(f,g)(x)=\int_{\R^{6}}
m(\xi,\eta)\widehat{f}(\xi)\widehat{g}(\eta)e^{ix(\xi+\eta)}d\xi
d\eta$.

Following \cite{GN, GLNW}, for $u$ and $N$, we use the Strichartz
norms with angular regularity:
\begin{align}
u\in X=&\jb{D}^{-1}(L^\I_tH^{0,s}_\omega \cap L_t^2\dot
B^{1/4+\e,s}_{(q(\e),\gamma(\e)-),\omega}\cap L_t^2\dot B^{0,s}_{6,\omega}),\label{Strz norms1}\\
N\in Y=&L^\I_tH^{0,s}_\omega \cap L^2_t\dot
B^{-1/4-\e,s}_{(q(-\e),2+),\omega},\label{Strz norms2}
\end{align}
for fixed $0<\e\ll 1$, where \EQ{
 \frac{1}{q(\e)}=\frac{1}{4}+\frac{\e}{3},\quad \frac{1}{\gamma(\e)}=\frac{3}{8}+\frac{5\e}{6}.}
The condition $0<\e\ll 1$ ensures that \EQ{
 \frac{10}{3}<q(\e)<4<q(-\e)<\I,}
such that the norms in \eqref{Strz norms1}-\eqref{Strz norms2}
satisfy the condition in Theorem \ref{thm1}.

\begin{lem}\label{lem:bilinear}
Let $1\leq p,p_1,p_2\leq \infty$ and $1/p=1/p_1+1/p_2$. Assume
$m(\xi,\eta)$ is a symbol on $\R^6$, $m(A \xi,A\eta)=m(\xi,\eta)$
for any $A\in SO(3)$, $m$ is bounded and satisfies for all
$\alpha,\beta$
\[|\partial_\xi^\alpha \partial_\eta^\beta m(\xi,\eta)|\leq C_{\alpha\beta}|\xi|^{-|\alpha|}|\eta|^{-|\beta|},\quad \xi,\eta \ne 0.\]
Then for $q>1$, $q_1,\tilde q_1\in (1,\infty)$, $\tilde q_1,\tilde
q_2 \in (1,\infty]$, and $1/q=1/{q_1}+1/{q_2}=1/{\tilde
q_1}+1/{\tilde q_2}$
\[\norm{T_m(P_{k_1}f,P_{k_2}g)}_{\Lr_r^{p}\Hl^\al_{q}}\les
\norm{f}_{\Lr_r^{p_1}\Hl^{\al}_{q_1}}\norm{g}_{{\Lr_r^{p_2}L_\omega^{q_2}}}+\norm{f}_{\Lr_r^{p_1}L_\omega^{\tilde
q_2}}\norm{g}_{{\Lr_r^{p_2}\Hl^{\al}_{\tilde q_1}}}\] holds for any
$k_1,k_2\in \Z$, with an uniform constant $C$.
\end{lem}

\begin{proof}
We can write
\[T_m(P_{k_1}f,P_{k_2}g)(x)=\int K(x-y,x-y')f(y)g(y')dydy'\]
where the kernel is given by $K(x,y)=\int
m(\xi,\eta)\chi_{k_1}(\xi)\chi_{k_2}(\eta)e^{ix\xi+iy\eta}d\xi
d\eta$. From the assumption on $m$, and integration by parts, we get
a pointwise bound of the kernel:
\begin{align}\label{eq:pointwiseK}
|K(x,y)|\les 2^{3k_1}(1+|2^{k_1}x|)^{-4}2^{3k_2}(1+|2^{k_2}y|)^{-4}.
\end{align}
Since $m(A\xi,A\eta)=m(\xi,\eta)$, then $K(Ax,Ay)=K(x,y)$ for any
$A\in SO(3)$. Then we have
\begin{align*}
\norm{T_m(P_{k_1}f,P_{k_2}g)}_{\Lr_r^{p}\Hl^\al_{q}}\les&
\norm{T_m(P_{k_1}f,P_{k_2}g)}_{\Lr_r^{p}L^q_{\omega}}+\norm{D_\omega^\alpha T_m(P_{k_1}f,P_{k_2}g)}_{\Lr_r^{p}L^q_{\omega}}\\
:=&I+II.
\end{align*}
We only consider the term $II$ since term $I$ can be handled in an
easier way. By the fractional derivative on $SO(3)$ (see the
appendix) we get
\begin{align*}
II=&\norm{A(D_\omega^\al T_m(P_{k_1}f,P_{k_2}g))}_{L_x^pL_A^q}=\norm{D_A^\al A(T_m(P_{k_1}f,P_{k_2}g))}_{L_x^pL_A^q}\\
=&\normo{ \int K(x-y,x-y')D_A^\al[f(Ay)g(Ay')]dydy'}_{L_x^pL_A^q}\\
\les&
\norm{f}_{\Lr_r^{p_1}\Hl^{\al}_{q_1}}\norm{g}_{{\Lr_r^{p_2}L_\omega^{q_2}}}+\norm{f}_{\Lr_r^{p_1}L_\omega^{\tilde
q_2}}\norm{g}_{{\Lr_r^{p_2}\Hl^{\al}_{\tilde q_1}}}
\end{align*}
where we used Lemma \ref{lem:SO3sob} in the last step.
\end{proof}

Now we follow the proof with slight modifications in \cite{GN} to
prove Theorem \ref{thm2}. It suffices to prove the nonlinear
estimates. Fix $s>3/4$. The following two lemmas can be proved
similarly as Lemma 3.2-3.3 in \cite{GN}. The main difference is that
we use Lemma \ref{lem:bilinear} for every bilinear dyadic piece.
\begin{lem}[Bilinear terms I]\label{lem:bi1}
(1) For any $N$ and $u$, we have
\begin{align*}
 \|(Nu)_{LH}\|_{L^1_tH^{1,s}_\omega} \les& \|N\|_{L^2_t\dot B^{-1/4-\e,s}_{(q(-\e),2+),\omega}}\|\jb{D}u\|_{L^2_t \dot
 B^{1/4+\e,s}_{(q(\e),\gamma(\e)-),\omega}},\\
 \|(Nu)_{HH}\|_{L^1_tH^{1,s}_\omega} \les& \|N\|_{L^2_t\dot B^{-1/4-\e,s}_{(q(-\e),2+),\omega}}\|\jb{D}u\|_{L^2_t \dot
 B^{1/4+\e,s}_{(q(\e),\gamma(\e)-),\omega}}.
\end{align*}

(2) If $0\leq \theta\leq 1$, $\frac{1}{\tilde
q}=\frac{1}{2}-\frac{\theta}{2}$, $\frac{1}{\tilde
r}=\frac{1}{4}+\frac{\theta}{3}+\frac{\e}{3}$, then for any $N$ and
$u$
\begin{align*}
\|(Nu)_{\alpha L}\|_{\LR{D}^{-1}L^{\tilde q'}_t \dot
B^{\frac{3}{2}-\frac{2}{\tilde q}-\frac{3}{\tilde r},s}_{(\tilde
r',2),\omega}} \lec \|N\|_{L^2_t\dot
B^{-1/4-\e,s}_{(q(-\e),2+),\omega}}\|u\|_{L_t^\infty
H_\omega^{0,1}\cap L^2_t \dot
 B^{1/4+\e,s}_{(q(\e),\gamma(\e)-),\omega}}.
\end{align*}
\end{lem}

\begin{lem}[Bilinear terms II]\label{lem:bi2} (1) For any $u$, we have
\EQ{
 \|D(u\bar u)_{HH}\|_{L^1_tH^{0,s}_\omega} \lec \|u\|_{L^2_t\dot B^{1/4-\e,s}_{(q(-\e),\gamma(\e)-),\omega}}\|\jb{D}u\|_{L^2_t \dot B^{1/4+\e,s}_{(q(\e),\gamma(\e)-),\omega}}.}

(2) If $0\leq \theta\leq 1$, $\frac{1}{\tilde
q}=\frac{1}{2}-\frac{\theta}{2}$, $\frac{1}{\tilde
r}=\frac{1}{4}+\frac{\theta}{3}-\frac{\e}{3}$, then \EQ{
 \|D(u\bar u)_{\alpha L+L\al}\|_{L^{\tilde q'}_t \dot
B^{\frac{3}{2}-\frac{1}{\tilde q}-\frac{3}{\tilde r},s}_{(\tilde
r',2),\omega}} \lec \|\jb{D}u\|_{L_t^\infty H_\omega^{0,s}\cap L^2_t
\dot B^{1/4+\e,1}_{(q(\e),\gamma(\e)-)},\omega}^2.}
\end{lem}

\begin{rem}
In application, we will use Lemma \ref{lem:bi1} (2) and Lemma
\ref{lem:bi2} (2) by fixing a $0<\theta_0\ll 1$ such that by this
choice $(\tilde q,\tilde r)$ is admissible to apply Corollary
\ref{cor:str}.
\end{rem}

For the boundary terms and cubic terms, we can use Lemma
\ref{lem:bilinear} to prove the estimates similarly as Lemma 3.5 and
Lemma 3.7 obtained in \cite{GN}.

\section{Appendix: Sobolev spaces on $SO(3)$}

In the appendix, we collect some facts about the Sobolev spaces on
$SO(3)$. It is well-known that $G=SO(3)$ is a compact Lie group,
with a Haar measure $\mu$. Let $X_1(t),X_2(t),X_3(t)$ denote the
subgroups with $x,y,z$-axis as axis of rotations, namely
\begin{align*}
\begin{pmatrix} 1&0&0\\0 &\cos t & -\sin t \\
0&\sin t & \cos t
\end{pmatrix},\,
\begin{pmatrix} \cos t & 0 & -\sin t \\
0 & 1 &0\\\sin t &0& \cos t
\end{pmatrix},\,
\begin{pmatrix} \cos t & -\sin t &0 \\
\sin t & \cos t &0\\ 0&0&1
\end{pmatrix},\, t\in \R.
\end{align*}
Let $X_j$ be the vector fields induced by $X_j(t)$, $j=1,2,3$,
namely for $f\in C^\infty (G)$
\[X_j(f)=\frac{d}{dt}\bigg{|}_{t=0}f(AX_j(t)).\]
Define the Laplacian operator $\Delta_A$ on $G$ by
$\Delta_A=\sum_jX_j^2$ and $D_A=(-\Delta_A)^{1/2}$. Then we define
the Sobolev space $H^s_p(G)$ by the norm
$\norm{f}_{H^s_p(G)}=\norm{f}_{L_\mu^p(G)}+\norm{D_A^{s}f}_{L_\mu^p(G)}$.

$\Des$ denotes the Laplace-Beltrami operator on the unite sphere
$\cir^{2}$ endowed with the standard metric $g$ and measure
$d\omega$ and $D_\omega=(-\Des)^{1/2}$. Let $\tilde
X_{1}=x_2\p_{x_3}-x_3\p_{x_2}$, $\tilde
X_{2}=x_3\p_{x_1}-x_1\p_{x_3}$, $\tilde
X_{3}=x_1\p_{x_2}-x_2\p_{x_1}$. It is well-known that for $f\in
C^2(\R^3)$
\[\Delta_\omega(f)(x)=\sum_{1\leq j\leq 3}\tilde X_{j}^2(f)(x).\]
For $f\in C^2(\R^3)$, $A\in SO(3)$, define the action $A(f)=f(Ax)$.
It is easy to see that $X_j[f(Ax)]=A(\tilde X_j(f))$, and hence
$(-\Delta_A)^s[f(Ax)]=A((-\Des)^sf)$. Moreover, we have
$\int_{\cir^2} f(x)d\omega=\int_Gf(Ax)d\mu$. Therefore, we can
transfer freely the fractional derivatives between $\cir^2$ and
$SO(3)$. We will use the fractional Leibniz rule on $SO(3)$ which
was proved in \cite{CRT}.
\begin{lem}[Theorem 4, \cite{CRT}]\label{lem:SO3sob}
Let $\alpha \geq 0$, $p_1,q_2\in (1,\infty]$ and $r,p_2,q_1\in
(1,\infty)$ such that $1/r=1/p_i+1/q_i,i=1,2$. Then
\[\norm{D_A^\alpha (fg)}_{L^r_\mu (G)}\les \norm{f}_{p_1}\norm{D_A^\al g}_{q_1}+\norm{g}_{q_2}\norm{D_A^\al f}_{p_2}.\]
\end{lem}

\subsection*{Acknowledgment}
The author would like to thank Kenji Nakanishi for the precious
suggestions and Xiaolong Han for helpful discussions. This work is
supported in part by NNSF of China (No.11371037) and ARC Discovery
Grant DP110102488.

\end{document}